\documentclass[11pt,letterpaper]{amsart}

\usepackage{amsmath,amssymb,amsthm}
\usepackage{enumerate}
\usepackage{multirow}
\usepackage{tikz}
\usepackage{amsfonts}

\usetikzlibrary{matrix,arrows,decorations.pathmorphing}

\renewcommand{\phi}{\varphi}

\newcommand{\F}{{\mathbb F}}

\newtheorem{theorem}{Theorem}[section]
\newtheorem{proposition}[theorem]{Proposition}

\newtheorem{corollary}[theorem]{Corollary}

\theoremstyle{definition}
\newtheorem{remark}[theorem]{Remark}
\newtheorem{definition}[theorem]{Definition}

\def\cqfd{
{\hfill
\kern 6pt\penalty 500
\raise -1pt\hbox{\vrule\vbox to 5pt{\hrule width 4pt
\vfill\hrule}\vrule}}
\break}

\def\Frac{\mathop{\rm Frac}\nolimits}

\font\tengoth=eufm10
\font\sevengoth=eufm7
\font\fivegoth=eufm5
\newfam\gothfam
\textfont\gothfam=\tengoth\scriptfont\gothfam=\sevengoth\scriptscriptfont\gothfam=\fivegoth

\title{On the maximum number of rational points on singular curves over finite fields}

\author[Toulon,Luminy]{Yves Aubry}
\address[Toulon]{Institut de Math\'ematiques de Toulon, Universit\'e de Toulon, France}
\address[Luminy]{Institut de Math\'ematiques de Marseille, CNRS-UMR 7373, Aix-Marseille Universit\'e, France}
\email{yves.aubry@univ-tln.fr}

\author[Luminy]{Annamaria Iezzi}
\address[Luminy]{Institut de Math\'ematiques de Marseille, CNRS-UMR 7373, Aix-Marseille Universit\'e, France}
\email{annamaria.iezzi@univ-amu.fr}

\begin{document} 

\begin{abstract}
We give a construction of singular curves with many rational points over finite fields. This construction enables us to prove some results on the maximum number of rational points on an absolutely irreducible  projective algebraic curve defined over ${\F}_q$ of geometric genus $g$ and arithmetic genus $\pi$.

\bigskip

Keywords: Singular curves, finite fields, rational points, zeta function.\\
MSC[2010] 14H20, 11G20, 14G15.

\end{abstract}


\maketitle

\bigskip
\begin{center}
{\it This paper is respectfully and affectionately dedicated to Mikha\"{\i}l Tsfasman and Serge Vl{\u a}du{\c t}, our colleagues and our friends.}
\end{center}
\bigskip

\section{Introduction}

Singular curves over finite fields arise naturally in many  mathematics problems. An example comes from coding theory with the geometric constructions of error correcting codes defined by the evaluation of points on algebraic varieties (see \cite{Little}). The study of hyperplane sections or more generally of sections of algebraic varieties is needed to find the fundamental parameters of these codes, and we often meet with singular varieties.
Another example arises from the theory of Boolean functions for which we have a geometric characterization of the Almost Perfect Nonlinear property   by determining whether the rational points on a certain algebraic set (which is a singular curve or a singular surface) are included in the union of hyperplanes (see \cite{A-R}).

Throughout the paper, the word \emph{curve} will always stand for an absolutely irreducible projective algebraic curve.

The zeta function of a singular curve defined over the finite field ${\mathbb F}_q$ with $q$ elements has been studied in \cite{A-P} and \cite{A-P1}. In particular,  in   \cite{A-P}  it is proved that, if $X$ is a curve defined over $\F_q$ of geometric genus $g$ and arithmetic genus $\pi$, then the number of rational points on $X$ satisfies:
$$\sharp{X}({\F}_q) \leq  q+1+gm+ \pi-g, \eqno(A)$$
where $m=[2\sqrt{q}]$ is the integer part of $2\sqrt{q}$. A curve attaining the bound $(A)$ will be called \emph{maximal}.

Fukasawa, Homma and Kim in \cite{FHK} proved that this bound is reached in the case where $g=0$ and $\pi=\frac{q^2-q}{2}$ by exhibiting a rational plane curve $B$ of degree $q+1$ with $q+1+\frac{q^2-q}{2}$ rational points.

The purpose of this paper is to study, in the general case, what is the maximum number of rational points on a singular curve. In order to do this, for $q$ a power of a prime, $g$ and $\pi$ non negative integers such that $\pi\geq g$,  we introduce  the quantity $N_q(g,\pi)$ defined as the maximum number of rational points over ${\F}_q$ on a curve defined over ${\mathbb F}_q$ of geometric genus $g$ and arithmetic genus $\pi$. We recall that  $N_q(g)$ is the usual notation for the maximum number of rational points over ${\F}_q$ on a smooth  curve defined over ${\mathbb F}_q$  of genus $g$. It follows from \cite{A-P} that:
$$N_q(g,\pi)\leq N_q(g)+\pi-g.\eqno(B)$$
A curve  defined over $\F_q$ of geometric genus $g$ and arithmetic genus $\pi$ with $N_q(g)+\pi-g$ rational points will be called $\delta$-\emph{optimal}.

Following the ideas of Rosenlicht in \cite{ROS} and Serre in Chap. IV of \cite{GACC}, we give in Theorem \ref{singu} a construction of singular curves defined over ${\mathbb F}_q$  with prescribed rational singularities which enables us to control the arithmetic genus.

Then we use this construction to  prove (Theorem \ref{principal}) that, if $X$ is a smooth curve of genus $g$ defined over $\mathbb F_q$ and if $\pi$ is an integer of the form
$$\pi=g+a_2+2a_3+3a_4+\cdots+(n-1)a_n$$
with $0\leq a_i\leq B_i(X)$, where $B_i(X)$ is the number of closed points of degree $i$ on the curve $X$, then there exists a curve $X'$ over $\mathbb F_q$ of arithmetic genus $\pi$ such that $X$ is the normalization of $X'$  and
$$\sharp{X'}(\F_q)=\sharp{X}(\F_q)+a_2+a_3+a_4+\cdots+a_n.$$

This allows us to show (Theorem \ref{iff}):
$$N_q(g,\pi)= N_q(g)+\pi-g$$
if and only if $g\leq \pi \leq g+B_2({\mathcal X}_q(g))$, where $B_2({\mathcal X}_q(g))$ denotes the maximum number of points of degree 2 on a smooth curve having $N_q(g)$ rational points over $\F_q$.
In particular, for $g=0$, it implies that the bound $(A)$ is reached if and only if $0\leq \pi\leq \frac{q^2-q}{2}$, showing that the curve $B$ in \cite{FHK} is  a very particular case of Theorem \ref{iff}.

Furthermore, we obtain in Corollary \ref{gen1} an explicit characterization of   $\delta$-optimal curves with $g=1$.

Finally, we deal with some properties of maximal curves.

\section{Notations and preliminary}

Let $X$ be a curve  defined over $\mathbb F_q$ with function field  $\mathbb F_q(X)$ and let $\tilde X$ be its normalization.

For every point $Q$ on $X$ we denote by  $\mathcal O_Q$ the local ring of $X$ at $Q$, by $\mathcal M_Q$  the maximal ideal of $\mathcal O_Q$ and by
$\deg Q =\left[\mathcal O_Q/\mathcal M_Q:\mathbb F_q\right]$ the degree of $Q$.

Let $\overline{\mathcal O_Q}$ be the integral closure of $\mathcal O_Q$ in $\mathbb F_q(X)$. The \emph{degree of singularity} of $Q$ on $X$ is defined by:
$$\delta_Q:=\dim_{\mathbb F_q}\overline{\mathcal O_Q}/\mathcal O_Q.$$
We remark that $\delta_Q=0$ if and only if $Q$ is a non-singular point. 

Now if we set:
$$\delta:=\sum_{Q\in \operatorname{Sing}X}\delta_Q,$$
with $\operatorname{Sing}X$ denoting the (finite) set of singular points on $X$, the \emph{arithmetic genus} $\pi$ of $X$ can be defined as (see Prop. 3 - Chap. IV of \cite{GACC}):
$$\pi:=g+\delta,$$
where  $g$ is the genus of $\tilde X$ (called the \emph{geometric genus} of $X$).

We have obviously $\pi\geq g$ and we have $\pi=g$ if and only if $X$ is a smooth curve.

\bigskip
In  \cite{A-P} the authors established some connections between a (singular) curve and its normalization, in terms of the number of rational points and the zeta function. Firstly they proved:
\begin{equation}\label{diff} \vert \sharp \tilde{X}({\F}_q)-\sharp{X}(\F_q)\vert \leq \pi-g.\end{equation}

Furthermore they proved that the zeta function $Z_X(T)$ of $X$  is the product of the zeta function $Z_{\tilde{X}}(T)$ of $\tilde{X}$ by a polynomial of degree $\Delta_X:=\sharp(\tilde{X}(\overline{\mathbb F}_q)\backslash X(\overline{\mathbb F}_q))\leq \pi-g$. 
More precisely, if $\nu:\tilde{X}\rightarrow X$ denotes the normalization map, they showed that:
 \begin{equation}\label{zetaf}
Z_X(T)=Z_{\tilde{X}}(T)\prod_{P\in \operatorname{Sing} X }\left(\frac{\prod_{Q\in \nu^{-1}(P)}(1-T^{\deg Q})}{1-T^{\deg P}}\right).
\end{equation}
As a consequence  they obtained that for all $n\geq 1$:
 $$\sharp X(\mathbb F_{q^n})= q^n+1-\sum_{i=1}^{2g}\omega_i^n -\sum_{j=1}^{\Delta_X}\beta_j^n,$$
 for some algebraic integers $\omega_i$ of absolute value $\sqrt{q}$ and some roots of unity $\beta_j$ in $\mathbb C$.
 
 In particular, they got the inequality:
\begin{equation}\label{YvesMarc}
\vert \sharp X({\F}_q) -(q+1)\vert \leq
gm+ \pi-g.
\end{equation}
The integer part $m=[2\sqrt{q}]$ comes from the Serre improvement of the Weil bound (see \cite{Ser}).

\bigskip

For  $g$ and $\pi$ non negative integers such that $\pi\geq g$, we define the quantity 
$N_q(g,\pi)$
as the maximum number of rational points over ${\F}_q$ on a  curve defined over ${\F}_q$ of geometric genus $g$ and arithmetic genus $\pi$.

We have clearly $N_q(g,g)=N_q(g),$
where $N_q(g)$ is the usual notation for the maximum number of rational points over ${\F}_q$ on a smooth   curve defined over ${\F}_q$ of genus $g$.
If $X$ is a  curve defined over ${\F}_q$ of geometric genus $g$ and arithmetic genus $\pi$, we obtain from (\ref{diff}): 
$\sharp{X}({\F}_q) \leq \sharp \tilde{X}({\F}_q)+\pi-g\leq N_q(g)+\pi-g.$
Therefore we have the following upper bound for $N_q(g,\pi)$:
$$N_q(g,\pi)\leq N_q(g)+\pi-g=N_q(g)+\delta \eqno(B)$$
and using the inequality (\ref{YvesMarc}), we have also:

$$N_q(g,\pi)\leq  q+1+gm+ \pi-g. \eqno(A')$$

A  question naturally arises:
\begin{center}\emph{For which $q, \, g$ and $\pi$ the bounds $(B)$ and $(A')$ are attained?  }

\end{center}

In order to answer the question,  we are going  to construct singular curves with prescribed geometric genus $g$ and  arithmetic genus $\pi$ and having ``many" rational points.


\section{Curves with prescribed singularities}

Let $X$ be a smooth curve over $\mathbb F_q$ with function field $\mathbb F_q(X)$ and let $S=\{Q_1,\ldots, Q_s\}$ be a non-empty finite set of closed points on $X$. We consider the following  subring  $\mathcal O \subset \mathbb F_q(X)$:
$$\mathcal O=\bigcap_{i=1}^s\mathcal O_{Q_i}.$$
The ring $\mathcal O$ is a finite intersection of discrete valuation rings. It is well known (see Prop. 3.2.9 in \cite{STI}) that $\mathcal O$  is a semi-local ring, its maximal ideals are precisely $\mathcal N_{Q_i} := \mathcal M_{Q_i} \cap \mathcal O$ for  $i=1,\ldots, s$ and the fields $\mathcal O_{Q_i}/\mathcal M_{Q_i}$ and $\mathcal O/\mathcal N_{Q_i}$ are isomorphic. Moreover $\mathcal O$ is a principal ideal domain and therefore a Dedekind domain (see Prop. 3.2.10 in \cite{STI}).

Now, let $n_1, \ldots , n_s$ be $s$ positive integers, set ${\mathcal N}:=\mathcal N_{Q_1}^{n_1}\cdots \mathcal N_{Q_s}^{n_s}$  and let us consider the subring $\mathcal O'\subseteq \mathcal O $ defined by:
\begin{equation}\label{notation}
\mathcal O':= \mathbb F_q + \mathcal N.
\end{equation}
Hence a function in $\mathcal O'$  takes the same value at the points $Q_1,\ldots, Q_s$.

\bigskip
We can represent the inclusions with the following diagram:
\begin{center}
\begin{tikzpicture}[description/.style={fill=white,inner sep=2pt}]
\matrix (m) [matrix of math nodes, row sep=2em,
column sep=3em, text height=1.5ex, text depth=0.25ex]
{  &\mathbb F_q(X)&\\
\mathcal O_{Q_1}& \cdots &  \mathcal O_{Q_s} \\
   &  \mathcal O&  \\
   & \mathcal O'&  \\ };
\path[-,font=\scriptsize]
(m-1-2) edge (m-2-1)
		edge	(m-2-2)
		edge	(m-2-3)
(m-2-1) edge (m-3-2)
(m-2-2) edge (m-3-2)
(m-2-3) edge (m-3-2)
(m-3-2) edge (m-4-2)
	    ;
\end{tikzpicture}

\end{center}

\begin{proposition} \label{local ring}$\mathcal O'$ verifies the following properties:
\begin{enumerate}
\item $\Frac(\mathcal O')=\mathbb F_q(X)$
and $\mathcal O$ is the integral closure of $\mathcal O'$ in $\mathbb F_q(X)$.
\item $\mathcal O'$ is a local ring with maximal ideal $\mathcal N$ and residual field  $\mathcal O'/\mathcal N\cong \mathbb F_q$. Moreover $\mathcal N$ is the conductor of $\mathcal O'$ in $\mathcal O$ and, by definition, it contains the functions of $\mathcal O'$ that vanish at the points $Q_1,\ldots, Q_s$.
\item $\mathcal O/\mathcal O'$ is an $\mathbb F_q$-vector space such that $$\dim_{\mathbb F_q} (\mathcal O/\mathcal O')=\sum_{i=1}^sn_i\deg{Q_i}-1.$$
\end{enumerate}
\end{proposition}
\begin{proof}
Let us prove the first assertion. We know from Prop. 3.2.5 in \cite{STI} that $\Frac(\mathcal O)= \mathbb F_q(X)$. So it is enough to show that $\mathcal O\subseteq \Frac(\mathcal O')$.

As $\mathcal O$ is a principal ideal domain, there exists $t\in \mathcal O $ such that  $\mathcal N=t\mathcal O$. So let $x\in \mathcal O$. We have
$x=\frac{tx^2}{tx}$
so that $x\in  \Frac(\mathcal O')$.
 
  The integral closure  $\overline {\mathcal O'}$ of $\mathcal O'$ in $\mathbb F_q(X)$ is given by the intersection of all valuation rings of $\mathbb F_q(X)$ which contain $\mathcal O'$. So $\overline{\mathcal O'} \subseteq \bigcap_{i=1}^s \mathcal O_{Q_i}=\mathcal O.$ Hence it is enough to show that there are no other valuation rings that contain $\mathcal O'$.

Let $\mathcal O_P$ be a valuation ring in $\mathbb F_q(X)$ different from $\mathcal O_{Q_1},\ldots,\mathcal O_{Q_s}$  and let $v_P,v_{Q_1},\ldots, v_{Q_s}$ be the corresponding valuations. By the Strong Approximation Theorem  (see Prop. 1.6.5 in \cite{STI}) we can find an element $x \in\mathbb F_q(X)$ such that
$v_P(x)=-1$ and
$v_{Q_i}(x)= n_i$ for every  $i=1,\ldots,s.$
This implies that $x\in \mathcal N\subseteq \mathcal O'$ and $x\notin \mathcal O_P$. We conclude that $\overline {\mathcal O'}=\mathcal O$.

We prove now the second assertion. First we show that $\mathcal N$ is a maximal ideal in $\mathcal O'$.
We have that $\mathcal O'\setminus \mathcal N=\{a+n : a\in \mathbb F_q^*$ and $n\in \mathcal N \}$. If $\mathcal N$ were not a maximal ideal, by Zorn's lemma there would exist a maximal ideal $\mathcal N'$ such that $\mathcal N\subsetneq \mathcal N' \subsetneq \mathcal O'$. Let $x\neq 0$ be in $(\mathcal O'\setminus \mathcal N)\cap \mathcal N'$. So $x= a+n$, where $ a\in \mathbb F_q^*$ and $n\in \mathcal N$. Hence $a=x-n \in\mathcal N'$ and $\mathcal N'=\mathcal O'$, which is a  contradiction. 

If there existed another maximal ideal in $\mathcal O'$, by the Going Up Theorem (see Th. 5.10 in \cite{A-M}), it would be the contraction of a maximal ideal in $\mathcal O$.
But for every $i=1,\ldots,s$,  we have $\mathcal N\subseteq \mathcal N_{Q_i}\cap \mathcal O'$ and, as $\mathcal N$ is a maximal ideal in $\mathcal O'$, we have that $\mathcal N= \mathcal N_{Q_i}\cap \mathcal O'$. Hence $\mathcal N$ is the only maximal ideal in $\mathcal O'$ and thus $\mathcal O'$ is a local ring. We have obviously $\mathcal O'/\mathcal N\cong \mathbb F_q$.

Moreover, as $\mathcal N$ is both an ideal of $\mathcal O$ and the maximal ideal of $\mathcal O'$, it is the conductor of $\mathcal O'$ in $\mathcal O$.

Let us prove now the third assertion. By the third isomorphism theorem of modules we have:
$$\mathcal O/\mathcal O' \cong \frac{\mathcal O/\mathcal N}{\mathcal O'/\mathcal N}.$$
Furthermore we have an isomorphism of $\mathbb F_q$-vector spaces:
$${\mathcal O}/{\mathcal N}=\frac{\mathcal O}{\mathcal N_{Q_1}^{n_1}\cdots\mathcal N_{Q_s}^{n_s}}\cong   \prod_{i=1}^{s}\left(\frac{\mathcal O_{Q_i}}{\mathcal M_{Q_i}}\right)^{n_i}.$$
Hence $\dim_{\mathbb F_q}\mathcal O/\mathcal N= \sum_{i=1}^{s}n_i \deg Q_i$
and thus
$$\dim_{\mathbb F_q} \mathcal O/\mathcal O'=\dim_{\mathbb F_q}\mathcal O/\mathcal N -\dim_{\mathbb F_q}\mathcal O'/\mathcal N= \sum_{i=1}^{s}n_i \deg Q_i - 1.$$
\end{proof}
The special case of Proposition $\ref{local ring}$ for which $S$ is a singleton $\{Q\}$ will be pivotal in the next section:

\begin{corollary}\label{centro}
Let  $\mathcal O_Q$  be a discrete valuation ring of $\mathbb F_q(X)$ with maximal ideal $\mathcal M_Q$. Then the ring $\mathcal O'_Q:= \mathbb F_q+\mathcal M_Q$ is a local ring contained in $\mathcal O_Q$ such that $\left[\mathcal O'_Q/\mathcal M_Q:\mathbb F_q\right]=1.$ Furthermore $\mathcal O_Q$ is the integral closure of  $\mathcal O'_Q$ and $\mathcal O_Q/\mathcal O'_Q$ is an $\mathbb F_q$-vector space of dimension $\deg Q-1$.  
\end{corollary}

\begin{remark}\label{biggest}
We remark that the ring $\mathcal O'_Q$, defined as in Corollary $\ref{centro}$, is  the biggest (according to inclusion) local ring contained in $\mathcal O_Q$ such that $\left[\mathcal O'_Q/\mathcal M_Q:\mathbb F_q\right]=1.$ In fact, let $\mathcal O$ be a local ring contained in $\mathcal O_Q$, with $\mathcal M$ as maximal ideal, such that $\left[\mathcal O/\mathcal M:\mathbb F_q\right]=1$. As $\mathcal O/\mathcal M\cong\mathbb F_q$, every element $x$ in $\mathcal O$ is of the form $x=a+m$, with $a\in \mathbb F_q$ and $m\in \mathcal M$. Thus $\mathcal O=\mathbb F_q+\mathcal M$. But  $\mathcal M= \mathcal M_Q\cap \mathcal O\subseteq \mathcal M_Q$ and hence $\mathcal O\subseteq \mathcal  O'_Q$.
\end{remark}

The operations on the valuation rings contained in $\mathbb F_q(X)$ correspond to operations on the points on $X$: roughly speaking,  ``glueing'' together $\mathcal O_{Q_1},\ldots,\mathcal O_{Q_s}$ in the local ring $\mathcal O'$ corresponds to ``glueing'' together the non-singular closed points $Q_1,\ldots, Q_s$ in a singular rational one. In this way, starting from a smooth curve $X$, we define a ``glued'' curve $X'$ which is biregularly equivalent to $X$ except at the ``glued'' points, and for which the ``glued'' non-singular points are replaced by a singularity of a specific degree of singularity.

Formally:

\begin{theorem}\label{singu}
Let $X$ be a smooth  curve of genus $g$ defined over $\mathbb F_q$, let $Q_1,\ldots, Q_s$ be closed points on $X$. Let $n_1,\ldots, n_s$ be positive integers and consider the local ring $\mathcal O'$ defined as in (\ref{notation}). 

Then there exists a curve $X'$ defined over $\mathbb F_q$, having $X$ as normalization, with only one singular  point $P$ whose local ring is the  prescribed local ring $\mathcal O'$. Moreover $P$ is a  rational point on $X'$ with a  degree of singularity equal to $\sum_{i=1}^sn_i\deg{Q_i}-1$ and $X'$ has arithmetic genus  $g +\sum_{i=1}^sn_i\deg{Q_i}-1.$ \end{theorem}

\begin{proof}
The existence of the curve $X'$ is proved in Th. 5 in \cite{ROS} and  Prop. 2 - Chap. IV in \cite{GACC}.
The degree of the singular point, its degree of singularity and the arithmetic genus of $X'$ come from Proposition \ref{local ring} and the definition of the arithmetic genus.
\end{proof}
We can remark that, by construction, the curve $X'$ has a number of rational points equal to $$\sharp X'(\mathbb F_q) =\sharp X(\mathbb F_q) - \sharp\{ Q_i, i=1,\ldots, s\textrm{ such that } Q_i \textrm{ is rational  over } \mathbb F_q \}+1.$$

\begin{remark}\label{plusdepoints}
In the previous theorem we limit ourselves to the construction of a curve with only one singularity, but there is nothing preventing us from constructing more singularities at the same time, by giving a finite number of prescribed local rings, no two of which have a place in common.
\end{remark}


\section{Singular curves with many rational points and small arithmetic genus}

The construction given in the previous section can be used to produce singular curves with prescribed geometric genus and arithmetic genus and with many rationals points.

Firstly, we have the following lower bound for the quantity $N_q(g,\pi)$.

\begin{proposition}\label{lower} For $q$ a power of a prime, $g$ and $\pi$ non negative integers such that $\pi\geq g$  we have:
$$N_q(g,\pi)\geq N_q(g).$$
\end{proposition}
\begin{proof}
Let $X$ be a smooth curve of genus $g$ defined over $\mathbb F_q$  with $N_q(g)$ rational points and let $Q$ be a rational point on $X$. Let us consider the local ring
$$\mathcal O':= \mathbb F_q+ (\mathcal M_Q)^{\pi-g+1}.$$
By Theorem $\ref{singu}$, there exists  a curve $X'$ defined over $\mathbb F_q$, having $X$ as normalization and  biregularly equivalent to $X$ except at the point $Q$, such that $X'$ contains a singular point $P$ corresponding to the prescribed local ring $\mathcal O'$.  Hence $\sharp X(\mathbb F_q)=\sharp X'(\mathbb F_q)=N_q(g)$ and $X'$ has geometric genus $g$ and arithmetic genus equal to $g + \pi-g+1-1=\pi$. It follows that $ N_q(g,\pi)\geq N_q(g)$.
\end{proof}

Now, let us 
state the following useful result. 
\begin{theorem}\label{principal}
Let $X$ be a smooth curve of genus $g$ defined over $\mathbb F_q$. Let $\pi$ be an integer of the form
$$\pi=g+a_2+2a_3+3a_4+\cdots+(n-1)a_n$$
with $0\leq a_i\leq B_i(X)$, where $B_i(X)$ is the number of closed points of degree $i$ on the curve $X$. Then there exists a curve $X'$ defined over $\mathbb F_q$ of arithmetic genus $\pi$ such that $X$ is the normalization of $X'$  and
$$\sharp{X'}(\F_q)=\sharp{X}(\F_q)+a_2+a_3+a_4+\cdots+a_n.$$

\end{theorem}

\begin{proof}
Let us take, for every $i \in\{2,\ldots, n\}$, $a_i$ closed points $Q_1^{(i)},Q_2^{(i)},\ldots,Q_{a_i}^{(i)}$
of degree $i$ on  $X$. 
We obtain a subset $S=\{Q_j^{(i)}, 2\leq i\leq n, \ 1\leq j\leq a_i\}$
of cardinality $|S|=a_2+a_3+\cdots+a_n$.

For every $Q \in S$ we set
 $\mathcal O'_Q:= \mathbb F_q+ \mathcal M_Q$ and  we have, by Corollary \ref{centro},   that
$\mathcal O'_Q$ is a  local ring with maximal ideal $\mathcal M_Q$ and 
$\left[{\mathcal O'_Q}/{\mathcal M_Q}:\mathbb F_q\right]=1.$
Moreover, 
 $\mathcal O_Q$ is the integral closure of  $\mathcal O'_Q$
 and $\mathcal O_Q/\mathcal O'_Q$ is an $\mathbb F_q$-vector space of dimension $\deg Q-1$.

By Theorem $\ref{singu}$ and Remark \ref{plusdepoints}, we get a curve $X'$ defined over ${\mathbb F}_q$ having $X$ as its normalization and such that:
$$\sharp{X'}(\F_q)=\sharp{X}(\F_q)+a_2+a_3+a_4+\cdots+a_n.$$

Furthermore:
$$\pi=g+\sum_{2\leq i\leq n} \sum_{1\leq j\leq a_i}(\deg Q_j^{(i)}-1)=g+\sum_{2\leq i\leq n} \sum_{1\leq j\leq a_i}(i-1)$$
and thus:
$$\pi=g+\sum_{2\leq i\leq n}(i-1)a_i.$$
\end{proof}

\begin{remark}\label{remark1}
 Roughly speaking,  Theorem $\ref{principal}$ shows that we can ``transform'' a point of degree $d$ on a smooth curve in a singular rational one, provided that we increase the value of the arithmetic genus by $d-1$. \end{remark}

\begin{remark}\label{best}
The construction described in the proof of Theorem $\ref{principal}$ is ``optimal'', meaning that it allows us to provide a curve with a large number of rational points compared with its arithmetic genus. More precisely,
if $X$ is a curve defined over $\mathbb F_q$ of geometric genus $g$, arithmetic genus $\pi$ and with $N$ rational points and if $\tilde{X}$ is its normalization, then there exists a curve $X'$  constructed as in the proof of Theorem \ref{principal} of arithmetic genus $\pi'\leq \pi$, with $N'\geq N$ rational points and whose normalization is $\tilde{X}$. Indeed from Remark $\ref{biggest}$, since $\mathcal O'_Q$ is the biggest local ring contained in $\mathcal O_Q$ such that $\left[\mathcal O'_Q/\mathcal M_Q:\mathbb F_q\right]=1$, we obtain that every point in $S$ is ``replaced'' by a rational singular one  with the smallest  possible degree of singularity.
\end{remark}


\section{On $\delta$-optimal curves}

We introduce the following terminology, as mentioned in the introduction:
\begin{definition} Let $X$ be a curve over $\mathbb F_q$ of geometric genus $g$ and arithmetic genus $\pi$. The curve $X$ is said to be:
\begin{itemize} 
\item[(i)]  an \emph{optimal curve} if
$$\sharp{X}({\F}_q)= N_q(g,\pi);$$
\item[(ii)] a \emph{$\delta$-optimal curve} if 
$$\sharp{X}({\F}_q)= N_q(g)+ \pi-g=N_q(g)+\delta;$$
\item[(iii)] a \emph{maximal curve} if it attains the bound $(A)$, that is if
$$\sharp{X}({\F}_q)= q+1+gm+ \pi-g.$$
\end{itemize}
\end{definition}
Obviously, we have that maximal curves are $\delta$-optimal curves, and $\delta$-optimal curves are optimal curves. Moreover, we find again the classical definitions of optimal curve and maximal curve when $X$ is smooth.

Firstly we give some properties of $\delta$-optimal curves.
\begin{proposition}\label{opt}
Let $X$ be a  curve of geometric genus $g$ and arithmetic genus $\pi$. If $X$ is $\delta$-optimal then 
its normalization $\tilde{X}$ is an optimal curve and
 all the singular points are rational with degree of singularity equal to $1$. Moreover,
 if $\nu$ denotes the normalization map and $P$ is a singular point on $X$, then $\nu^{-1}(P)=\{Q\}$, with $Q$  a point of degree $2$ on $\tilde{X}$. Furthermore, we have
 $\pi-g\leq B_2(\tilde{X})$
and $Z_X(T)=Z_{\tilde{X}}(T)(1+T)^{\pi-g}$.

\end{proposition}

\begin{proof}
If $\tilde{X}$ was not an optimal curve, that is $\sharp\tilde{X}(\mathbb F_q)<N_q(g)$, we would have $\sharp X(\mathbb F_q)-\sharp\tilde{X}(\mathbb F_q)>\pi-g,$ which is  absurd by $(\ref{diff})$. 

 Hence we get:
$$\sharp X(\mathbb F_q)-\sharp\tilde{X}(\mathbb F_q)=\pi-g.$$
On the other hand we know from \cite{A-P} that:
$$\sharp X(\mathbb F_q)-\sharp\tilde{X}(\mathbb F_q)\leq \sharp \operatorname{Sing} X(\mathbb F_q)\leq \sharp \operatorname{Sing} X(\overline{\mathbb F}_q)\leq \pi-g$$
so that 
$$\sharp\operatorname{Sing} X(\mathbb F_q)=\sharp\operatorname{Sing} X(\overline{\mathbb F}_q)=\pi-g.$$
In particular this means that all singular points on $X$ are rational.

Moreover as $$\pi-g=\sum_{P\in \operatorname{Sing} X} \delta_P,$$ we find that $\delta_P=1$ for all $P\in \operatorname{Sing} X$.

Now,  let us look at the zeta function of $X$, which by $(\ref{zetaf})$ is given by:
$$Z_X(T)=Z_{\tilde{X}}(T)\prod_{P\in \operatorname{Sing} X}\left(\frac{\prod_{Q\in \nu^{-1}(P)}(1-T^{\deg Q})}{1-T}\right).$$
Since in our case 
$$\deg \left(\prod_{P\in \operatorname{Sing} X }\left(\frac{\prod_{Q\in \nu^{-1}(P)}(1-T^{\deg Q})}{1-T}\right)\right)=\pi-g,$$
then for every $P\in  \operatorname{Sing} X$, we have $\nu^{-1}(P)=\{Q\}$ with $\deg Q=2$ (because of the number of rational points on $X$, we can exclude the possibility for which $\nu^{-1}(P)=\{Q_1,Q_2\}$ with  $\deg Q_1=\deg Q_2=1$). 

In particular, we obtain, firstly,  that  $\sharp \operatorname{Sing}X\leq B_2(\tilde{X})$ which implies  $\pi-g\leq B_2(\tilde{X})$, and secondly
the zeta function of $X$.
\end{proof}

We would like, now, to determine for which values of $q$, $g$ and $\pi$ there exist $\delta$-optimal curves.

Remarking that the last quantity in  Theorem $\ref{principal}$ can be written in the form 
$$\sharp{X}(\F_q)+\pi - g -(a_3+2a_4+\cdots+(n-2)a_n),$$
and considering  Remark $\ref{remark1}$ and Proposition \ref{opt}, we find that the only  possibility to obtain a $\delta$-optimal curve is to start from an optimal smooth curve of genus $g$ and work only with its points of degree $2$: in fact for every point of degree $2$ that we ``transform'' in a rational one, the arithmetic genus increases  only by $1$. In this way we are able to construct a curve for which the difference between the number of rational points on it and on its normalization is exactly equal to the difference between the arithmetic genus and the geometric one. Obviously we are limited by the fact that the number of closed points of degree $2$ is finite, ``quite little'' for an optimal smooth curve and even sometimes equal to zero.

For these reasons we  introduce the following notation. Let us denote by ${\mathcal X}_q(g)$ the set of optimal smooth curves $X$ defined over $\F_q$ of genus $g$. Let $B_2({\mathcal X}_q(g))$ be the maximum number of points of degree 2 on a curve of ${\mathcal X}_q(g)$.

\begin{theorem}\label{iff} We have: 
$$N_q(g,\pi)= N_q(g)+\pi-g
\ \ \Longleftrightarrow\ \  g\leq \pi \leq g+B_2({\mathcal X}_q(g)).$$
\end{theorem}

\begin{proof}
Let $X$ be a curve in ${\mathcal X}_q(g)$. Let $\pi$ be an integer of the form
$\pi=g+a_2$
with $0\leq a_2\leq B_2(X)$, where $B_2(X)$ is the number of closed points of degree $2$ on the curve $X$. Then, by Theorem \ref{principal}, there exists a curve $X'$ over $\mathbb F_q$ of arithmetic genus $\pi$ such that $X$ is the normalization of $X'$ and
$$\sharp{X'}(\F_q)=\sharp{X}(\F_q)+a_2= N_q(g)+a_2.$$
Thus, for every $g\leq \pi \leq g+B_2({\mathcal X}_q(g))$, we have
 $N_q(g,\pi)=N_q(g)+\pi-g.$
 
The converse follows by Proposition \ref{opt}. 
\end{proof}

In the case where $g=0$, i.e. of rational curves,  Theorem $\ref{iff}$ turns in the following corollary:

\begin{corollary}\label{cor1} We have
$$N_q(0,\pi)=q+1+\pi
$$
 if and only if $\pi \leq \frac{q^2-q}{2}.$
\end{corollary}
\begin{proof}
The number of closed points of degree 2 of $\mathbb P^1$ is  given by 
$B_2(\mathbb P^1)=\frac{q^2-q}{2}.$
\end{proof}

In \cite{FHK},  Fukasawa, Homma and Kim study the rational plane  curve $B$ of degree $q+1$ defined as the image of 
$${\mathbb P}^1 \ni (s,t) \longmapsto (s^{q+1},s^qt+st^q,t^{q+1})\in {\mathbb P}^2.$$
They proved that the number of rational points on $B$ over ${\mathbb F}_q$ is $q+1+\frac{q^2-q}{2}$. 
Since its arithmetic genus is $\frac{q^2-q}{2}$, it follows that
$B$ is a maximal curve.
It is an explicit example of a rational singular curve  attaining $N_q(0,\frac{q^2-q}{2})$. 
The previous Corollary shows that $\delta$-optimal rational curves exist in fact for any $\pi \leq \frac{q^2-q}{2}.$

\bigskip

For the case of  curves of geometric genus $1$, it is well-known (see for example \cite{WAT}) that 
the number $N_q(1)$ is either $q+1+m$ or $ q+m$. It is $q+1+m$ if and only if at least one of the following occurs: $p$ does not divide $m$, or $q$ is a square, or $q=p$.

Then Theorem \ref{iff} implies:

 \begin{corollary} \label{gen1}\begin{enumerate}
\item If $p$ does not divide $m$, or $q$ is a square, or $q=p$ we have:
$$N_q(1,\pi)= q+1+m+\pi-1$$
if and only if $1 \leq \pi \leq 1+\frac{q^2+q-m(m+1)}{2}$.

\item In the other cases we have
$$N_q(1,\pi)= q+m+\pi-1$$
if and only if $1 \leq \pi \leq 1+\frac{q^2+q+m(1-m)}{2}$.
\end{enumerate}

\end{corollary}

\begin{proof}
 In the first case, we have that $N_q(1)=q+1+m$. So let $X$ be a smooth curve with  $q+1+m$ rational points. If we denote by $\omega$ and $\bar\omega$ the inverse roots of the numerator of the zeta function of $X$, then we have:
\begin{equation}\label{zeta1}\sharp X(\mathbb F_{q})=q+1+m= q+1-(\omega+\bar{\omega}),\end{equation}
and 
$$\sharp X(\mathbb F_{q^2})=q^2+1-(\omega^2+\bar{\omega}^2).$$
From $(\ref{zeta1})$ we obtain that $\omega+\bar{\omega}= -m$. Since $\omega$ has absolute value $\sqrt{q}$, we have:
$\omega^2+\bar{\omega}^2=(\omega+\bar{\omega})^2-2\omega\bar{\omega}=(\omega+\bar{\omega})^2-2|\omega|^2=m^2-2q,$
so that the number of rational points on $X$ over $\mathbb F_{q^2}$ is
$\sharp X(\mathbb F_{q^2})=q^2+1-(m^2-2q).$
In this way we obtain that for every maximal smooth curve of genus $1$ the number of points of degree 2 is equal to:
$$
B_2(X)=\frac{\sharp X(\mathbb F_{q^2} )- \sharp X(\mathbb F_{q})}{2}=\frac{q^2+q-m(m+1)}{2}$$
The conclusion follows from Theorem \ref{iff}.

For the second case, the proof is totally analogous to the previous one.  
\end{proof}
\begin{remark}
We remark that  for $q=2$, $q=3$ or $q=4$, as the  quantity $\frac{q^2+q-m(m+1)}{2}$ is equal to $0$, there exists no $\delta$-optimal curve over $\mathbb F_q$ of geometric genus $1$ and arithmetic genus $\pi>1$.
\end{remark}

\begin{remark}
We have already  seen (Proposition $\ref{opt}$) that a $\delta$-optimal (singular) curve has necessarily an optimal normalization. Nevertheless this is not in general true for an optimal singular curve. Let us show this fact with an example.

We consider curves over $\mathbb F_2$ of geometric genus $1$ and arithmetic genus $3$ and  we show firstly that $N_2(1,3)=6$. Indeed, we have $N_2(1,3)\geq N_2(1)=5$ by Proposition $\ref{lower}$ and
$N_2(1,3)<7$ by Corollary $\ref{gen1}$. 
Let now $X$ be a smooth curve of genus 1 over $\mathbb F_2$ with exactly $4$ rational points (its existence is assured by Th. 4.1 in \cite{WAT}). It is easy to show that such a curve has $3$ points of degree $2$. Applying  Theorem $\ref{principal}$ with $a_2=2$, we obtain a singular curve $X'$ over $\mathbb F_2$ of geometric genus $1$ and arithmetic genus $3$ with $6$ rational points and $X$ as normalization. So $N_2(1,3)=6$ and $X'$ is an example of an optimal singular curve   whose normalization is not optimal. 

Actually there is no optimal curve over $\mathbb F_2$ of geometric genus $1$ and arithmetic genus $3$ whose normalization is optimal. This is a consequence of the fact that an optimal smooth curve of genus $1$ has neither points of degree $2$ nor points of degree $3$.
\end{remark}

\vskip1cm

Finally we  focus on  the genera and the zeta function of maximal curves which form a particular subclass of $\delta$-optimal curves. 

We recall that for a maximal smooth curve $X$ defined over $\mathbb F_q$, it is known from Ihara in \cite{Iha} that if $q$ is a square the genus $g$ of $X$ verifies $g\leq \frac{q-\sqrt{q}}{2}$  (see also Prop. 5.3.3 of \cite{STI}). 

Moreover in this case ($q$ square), if $\omega_1, {\bar\omega_1},\ldots, \omega_g,{\bar\omega_g}$ are the inverse roots of the numerator  of the zeta function $Z_X(T)$ of $X$, then the maximality of $X$ with the Riemann hypothesis ($\vert \omega_i\vert=\sqrt{q}$) imply that they are all equal to $-\sqrt{q}$. Hence one gets:
$$Z_X(T)=\frac{(1+\sqrt{q}T)^{2g}}{(1-T)(1-qT)}.$$
If $q$ is not a square,  following an idea of Serre in \cite{Ser}, the arithmetic-geometric mean inequality  gives
$$\frac{1}{g}\sum_{i=1}^g(m+1+\omega_i+{\bar\omega_i})\geq \left(\prod_{i=1}^g(m+1+\omega_i+{\bar\omega_i})\right)^{1/g}\geq 1.$$
The maximality of $X$ implies that the arithmetic mean equals the geometric one, hence we find that $\omega_i+{\bar\omega_i}=-m$ for all $1\leq i\leq g$. Thus 
  $$Z_X(T)=\frac{(qT^2+mT+1)^g}{(1-T)(1-qT)}.$$

The following proposition shows what happens analogously for  general (i.e. possibly singular) maximal curves.

\begin{proposition}\label{max}
If $X$ is a maximal curve defined over ${\mathbb F}_q$ of geometric genus $g$, arithmetic genus $\pi$ and zeta function $Z_X(T)$, then we have:
$$\pi\leq g+ \frac{q^2+(2g-1)q-gm(m+1)}{2}$$
and
$$Z_X(T)=\frac{(qT^2+mT+1)^g(1+T)^{\pi-g}}{(1-T)(1-qT)}.$$
\end{proposition}
 
 \begin{proof}
 Let $X$ be a maximal curve  over ${\mathbb F}_q$ of geometric genus $g$ and arithmetic genus $\pi$. By Proposition $\ref{opt}$,  the curve $X$ has a maximal normalization $\tilde{X}$ and  $\pi\leq g+B_2(\tilde{X})$. But the number of points of degree $2$ is the same for all smooth maximal curves of genus $g$.
 Indeed, a smooth maximal curve over $\mathbb F_q$ of genus $g$ has, by definition, $q+1+gm$ rational points and thus
 the reciprocal  roots of the numerator polynomial of its zeta function are  two complex conjugate numbers $\omega$ and $\bar{\omega}$, each one with multiplicity $g$, such that $\omega+\bar{\omega}=-m$ and $\omega\bar{\omega}=q$.  With the same technique as in the proof of Corollary $\ref{gen1}$, we obtain:
$$
 B_2(\tilde{X})=\frac{q^2+(2g-1)q-gm(m+1)}{2}
$$
which gives the desired inequality involving $g$, $\pi$ and $q$.

From Proposition $\ref{opt}$ and the results on the zeta function of smooth maximal curves recalled above, we  find
directly the form of $Z_X(T)$.
 \end{proof}

\vskip1cm

{\it This work has been carried out in the framework of the Labex Archim\`ede (ANR-11-LABX-0033) and of the A*MIDEX project (ANR-11-IDEX-0001-02), funded by the ``Investissements d'Avenir" French Government programme managed by the French National Research Agency (ANR).}

\vskip1cm
\noindent
{\bf References:}
\bigskip

\bibliography{biblio} 
\bibliographystyle{plain} 
 
\end{document}